\documentclass[11pt]{article}

\usepackage{multicol,hyperref}
\usepackage{graphicx,epsfig,wrapfig}
\usepackage{amsmath,amssymb,amsfonts,amsthm,mathrsfs}
\usepackage[usenames,dvipsnames]{color}
\usepackage[ruled]{algorithm2e}

\theoremstyle{plain}
\newtheorem{thm}{Theorem}[section]
\newtheorem{lem}[thm]{Lemma}

\newtheorem{coro}[thm]{Corollary}
\theoremstyle{definition}
\newtheorem{defn}{Definition}

\theoremstyle{remark}

\newcommand{\probability}[1]{	\mathbb{P}\left\{#1\right\}}
\newcommand{\expectation}[1]{	\mathbb{E}\left[#1\right]}

\DeclareMathOperator{\sgn}{sgn}

\newcommand{\N}{\mathbb{N}}
\newcommand{\R}{\mathbb{R}}
\newcommand{\eps}{\epsilon}

\title{\small\bf HEAVY TRAFFIC LIMIT FOR THE WORKLOAD\\ PLATEAU PROCESS IN A TANDEM QUEUE\\ WITH IDENTICAL SERVICE TIMES}

\author{ {\small\sc H.\ Christian Gromoll, Bryce Terwilliger, Bert Zwart} \\
{\em\footnotesize University of Virginia and CWI} }

\begin{document}
\maketitle

\begin{abstract}
We consider a two-node tandem queueing network in which the upstream queue has
renewal arrivals with generally distributed service times, and each job reuses
its upstream service requirement when moving to the downstream queue. Both
servers employ the first-in-first-out policy. The reuse of service times
creates strong dependence at the second queue, making its workload difficult
to analyze. To investigate the evolution of workload in the second queue, we
introduce and study a process $M$, called the plateau process, which encodes
most of the information in the workload process. We focus on the case of
infinite-variance service times and show that under appropriate scaling,
workload in the first queue converges, and although the workload in the second
queue does not converge, the plateau process does converges to a limit $M^*$
that is a certain function of two independent L\'evy processes. Using
excursion theory, we derive some useful properties of $M^*$ and compare a time
changed version of it to a limit process derived in previous work. 

\end{abstract}



\section{Introduction}
The goal of this paper is to establish a stochastic process limit of a
two-node tandem queueing network where the first queue is a $GI/GI/1$ queue
(that is jobs have independent generally distributed service times and
independent generally distributed interarrival times) but in contrast to most
queueing models, customers reuse their specific service requirement when
moving to the second queue.  In other words, once a job's random service
requirement has been generated at the first queue, it will also be that job's
requirement at the second queue.  Both servers process jobs in
first-in-first-out order.

This structure induces a strong dependence between arrivals and services at
the second queue, leading to unusual phenomena and making even simple
performance measures such as the workload difficult to analyze. 

To visualize the effect of identical service times, consider the workload in
the second queue over a generic period during which both queues are busy.
During a given interarrival time for the second queue, its workload will
decrease by exactly the duration of this interarrival time (since we're
assuming the second queue does not empty during this period). But this time
equals the interdeparture time from the first queue, which equals the service
time of the job about to transfer.  Since this job reuses its service time at
the second queue, this also equals the amount of work about to enter the
second queue. So the workload in the second queue simply decreases by this
job's service time and then increases by the same amount when the job
transfers. The effect over a busy period of the second queue is a series of
returns to the same level attained at the previous arrival time.

This continues until a job in service at the first queue is larger than any
previous job in the first queue's busy period.  The workload in the second
queue will then empty and be zero for a while until the job transfers, at
which time the workload will increase to a new level that is higher than the
previous level, and resume a series of returns to this new level until the
next record-setting job comes through. 

Thus, during a busy period of the first queue, the workload in the second queue
is characterized by oscillations below a series of increasing levels or
plateaus. When the first queue experiences a period of idleness, this pattern
in the second queue is interrupted and its workload can reset to a new
starting height for the next series of plateaus. 

The pattern of frequent returns to the same level can be seen in Figure
\ref{fig:ThreeGraphs}, where the workload in the second queue must hit zero
before each level increase.  When compared visually to the workload in the
first queue, it is clear that the behavior is very different because the
workload in the second queue has frequent consecutive local maxima of the same
value, interspersed with occasional increases of that value. 

\begin{figure}\label{fig:ThreeGraphs}
\noindent\includegraphics[width =5.2in]{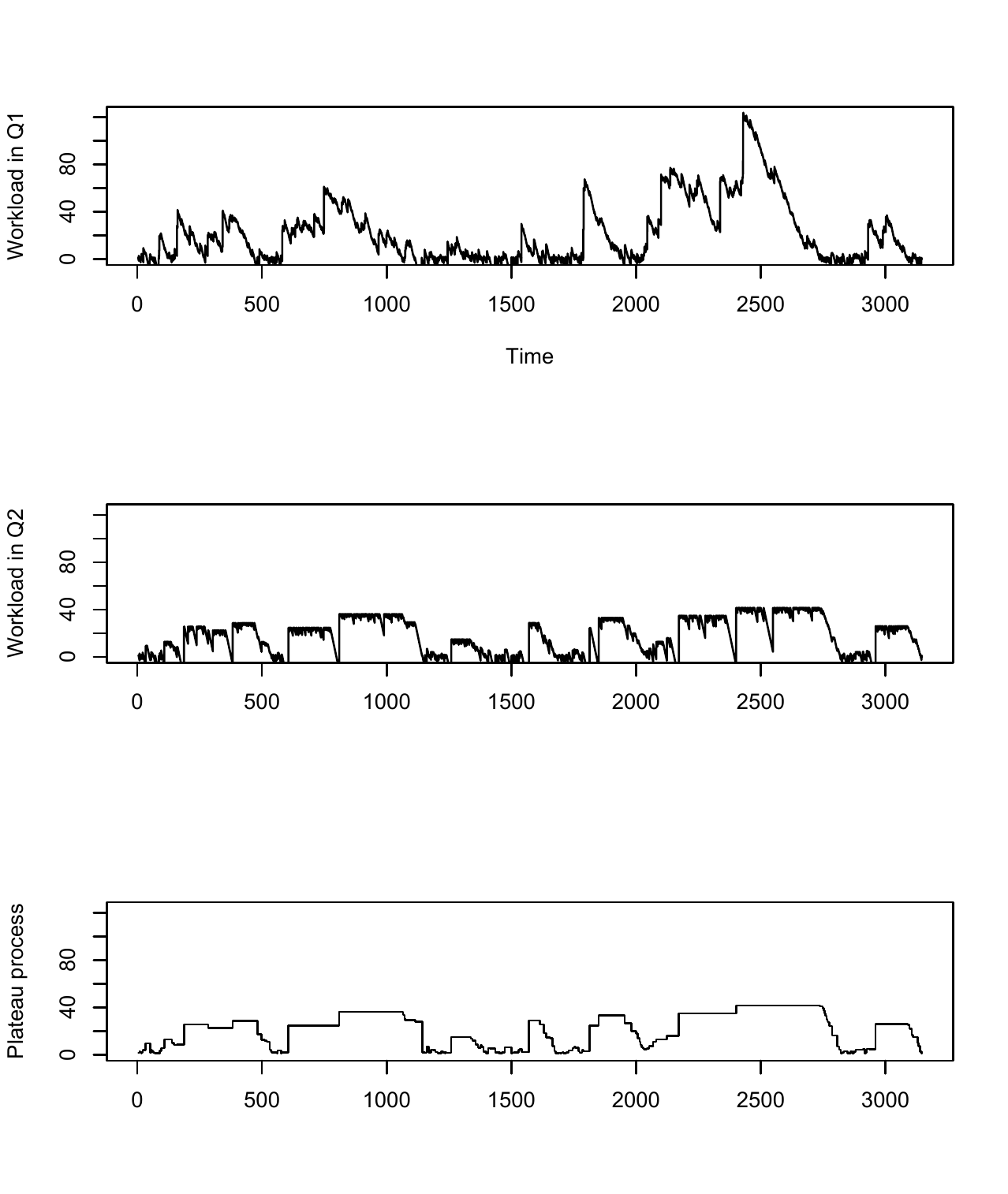}
\caption{The workload in both queues with identical service times in each.
1000 Poisson arrivals with parameter 1/3.1 service times are Pareto(1,3/2).}
\end{figure}

Why study such a model? After all, most queueing models in the literature 
make the Jacksonian assumption that jobs generate new independent service
times at each queue, and there are good reasons for this. For one, the
independence assumption is crucial to the mathematical techniques most often
employed, for example for deriving product-form descriptions of the steady state in
Jackson networks, or for proving diffusion approximations in generalized
Jackson networks. 

A second reason is that independence is a natural assumption in many
applications. Consider an automobile assembly line for example, where it would
make sense to assume that the time to attach the doors is independent of the
time to apply rust protection.

On the other hand, if we consider a manual automobile washing operation, it
seems natural that the main factor influencing service times is the soil level
of the vehicle. A very dirty vehicle will tend to have a longer service time
than others at the first washing station, but probably also at the wheel
washing station and very likely as well at the interior cleaning station.
That is, one would expect a vehicle's random service times at various
stations to be correlated with its soil level and thus to each other. 

Computer and telecommunications networks afford further examples in which jobs
must pass through a series of processing queues (transmission,
decryption, format translation et al.), the random processing time of which
will be correlated to the job's intrinsic size (file size).  Indeed one can
imagine many applications in which the successive service times of a given job
are highly correlated due to some intrinsic property of the job, and this
motivates consideration of models with correlated service times. 

While we are not proposing that the model studied here, with just two nodes
and identical service times is realistic for direct applications (such a model
would allow for a more general network topology and arbitrarily correlated as
opposed to {\em identical} service times), we view it as an archetype for more
realistic models incorporating correlation. It is the simplest possible model
in which the unusual effects of strong service time correlations are laid
bare, and yet it already exhibits the serious difficulties in analyzing such
effects.  Our aim is to demonstrate some useful mathematical tools for dealing
with such difficulties  (adding to the small handful of results that exist for
this model). We also speculate that some of the tools used here may be of use
in analyzing non-queueing models incorporating similar correlation structures,
such as models of world record evolution in improving populations as studied
in \cite{BalleriniResnick}.

The tandem model under consideration was first introduced for
Poisson arrivals in the PhD thesis of O.\ Boxma \cite{boxmathesis} where a
rather complete analysis of the invariant distribution was given, providing a
rare example of a non-product form tandem queueing network for which an
explicit analysis of the downstream queue is possible.

In addition,  this model also shows unusual behavior in heavy traffic. In the
finite variance case it is known \cite{KarpelevichKreinin, kk96} that the
amount of work at the second node is of smaller order than the amount of work
at the first node as the system load $\rho$ (which is identical for both
queues) tends to $1$. For service times with bounded support, it is even shown
in \cite{boxma1978longest} that the expected value of the waiting time in the
second queue is finite for $\rho=1$. The intuition behind these results is
that the amount of work in the first queue is driven by sums, and in the
second queue is driven by maxima, suggesting that both queues should scale
identically when the service times have infinite variance.

This behavior has recently been confirmed in our work \cite{GTZ1}, which is a
prequel to the present study. In \cite{GTZ1} we investigated the behavior of
the workload of the second queue  at embedded time points when the first queue
empties.  It was shown that this embedded Markov chain is sufficiently
tractable, and analytic methods were used to investigate the process limit of
this embedded Markov chain.

In the present paper we take up the task of analyzing the full workload
process at the second node. We seek to prove a scaling limit theorem wherein
the limit process is more tractable than the original workload process and can
therefore serve as an approximation to it.  One challenge is that the workload
process does not converge in heavy traffic.

To see this intuitively, consider that on a space-time scale under which the
successive plateaus of the process converge, there will be asymptotically
infinitely many arrivals in between plateau increases. Under scaling, each
such arrival causes a linear decrease at rate tending to infinity, followed by
an upward jump the same size as the total decrease. The asymptotic result is
oscillations below the level of the plateau that are too wild to converge in
any of the Skorohod topologies. 

We note that this type of behavior has been mentioned in Whitt's monograph
\cite{whitt2002stochastic}, where new spaces ($E$ and $F$) to potentially deal
with such fluctuations have been suggested.  Though an approach using this
framework would be interesting, we take a different approach in the present
paper which is more tailored to the specific model here.  Notice that the
silhouette of the workload in the second queue seems like it might converge in
the usual $J_1$-topology under the
same scaling as the workload in the first queue.  Moreover, much of the
information about the workload in the second queue is retained if we only keep
track of these recurring levels or plateaus, so we don't lose much by working
with just the silhouette. For example, if one is interested in the probability
of the buffer at the second queue exceeding some critical threshold, the
answer is the same for the silhouette. The silhouette also provides an upper
bound for the actual workload at any time, provides information about how
often the second queue is idle, etc. 

In choosing to work with the silhouette, we eliminate the oscillating behavior
that prevents us from working directly with the workload in the second queue,
and gain the ability to prove a limit theorem. 

This is the strategy we follow.  We introduce and study a process $M$, called
the plateau process, which encodes most of the information in the workload
process.  The plateau process is defined to be the workload in the second
queue at the time of the most recent arrival.  This definition eliminates the
difficulty with scaling described above. We show that under an appropriate
scaling, the plateau process converges to a limit $M^*$ that is a
certain function of two independent L\'evy processes $U^*$ and $V^*$.

More explicitly, the $N$th job waits in the second queue for a period of time
$F(U,V,1)(N)$, where $U$ and $V$ are the arrival and service processes for the
model, and for two functions $x,y:[0,\infty)\to\R$,
\begin{equation*}
\begin{split}
F(x,y,c)(t)&=\sup_{0\leq s\leq t}\left(y(s)-y(s-)+\sup_{0\leq r\leq s}\left(x(r)-y([r-c]^+\right)\right)\\
&\qquad\qquad-\left.\left.\sup_{0\leq s\leq t}\right(x(s)-y([s-c]^+)\right).
\end{split}
\end{equation*}
At time $t$ the number of jobs that have arrived to the second queue is
$R(t)$, and the above composition of functions is continuous on a relevant set in the
Skorohod path space $\mathbb D$.  For a sequence of models indexed by $r$, the
plateau process in the $r$th model can be written
\begin{equation*}
M^r(t)=F(U^r,V^r,1)(R^r(t)).
\end{equation*}
Letting $\check M^r(t)=\frac{1}{a_r}M^r(rt)$, we show that
\begin{equation*}
\check M^r\Rightarrow M^*,
\end{equation*}
where $M^*(t)=F(U^*+\gamma\mu e,V^*,0)(t/\mu)$; see Theorem \ref{Main Theorem}
below.

The process appearing in the limit is not Markovian, but a suitable
time-change is shown to be. Our second result provides (for a subset of cases) 
a means of performing some calculations on the limit process $M^*$, by
deriving an explicit formula for the one-dimensional distributions of a
natural time change $\{M^*(\mu L^{-1}(v)),v\ge0\}$ of the process. Here $\mu$
is a constant and $L^{-1}$ is
the inverse local time of a reflected version of the limiting service process
$V^*$, which is an explicit $\alpha$-stable L\'evy process. These
one-dimensional distributions are given for each $v\ge0$ by the distribution
functions
\[ 
  F_v(y)=\exp\left(-\int_y^{y+v}\frac{\kappa(q)}{q}dq\right), \qquad y\ge0,
  \]
where $\kappa$ is an explicit function; see Theorem \ref{thrm2main} in Section
\ref{s.LimitAnalysis}.  This second result also implies that the embedded
Markov chain of the limit process coincides with the limit of the embedded
Markov chains considered in \cite{GTZ1}.

The paper is organized is follows. We first carefully define the model and
scaling, make mild asymptotic assumptions, and state our first result, Theorem
\ref{Main Theorem}. 

The bulk of the paper, Sections \ref{s.Relationships} and \ref{s.Continuity},
is then devoted to the proof, which is essentially an elaborate application of
the continuous mapping theorem.  This is a bit delicate because the relevant
mapping $F$ is not continuous everywhere. In particular, we first show in a
series of steps that $M^r(t)$ can indeed be represented as the composition of
functions $F(U^r,V^r,1)(R^r(t))$ described above.  Then a series of steps
shows that $F$ is continuous on a particular subset of
$\mathbb{D}\times\mathbb{D}\times\mathbb{R}$ (see Lemma \ref{FisContinuous}).
Proving that for the limiting primitive processes $U^*$ and $V^*$, the triple
$(U^*+\gamma\mu e,V^*,0)$ is almost surely in this set enables a final
application of the continuous mapping theorem together with the random time
change theorem.  

Finally, in Section \ref{s.LimitAnalysis} we develop some ideas from excursion
theory to analyze the limit process $M^*$ for a subset of cases (when the
interarrival times have finite second moment). After performing a time
change using a local time derived from $V^*$, the process becomes Markov.  We
are then able to apply some excursion theory results to calculate one
dimensional distributions and relate this process to the limit derived in
\cite{GTZ1}.

\subsection{Notation}  The following notation will be used throughout.  Let
$\N=\{1,2,\ldots\}$ and let $\R$ denote the real numbers.  Let
$\R_{+}=[0,\infty)$.  For $a,b\in\R$, write $a\vee b$ for the maximum, and
  $a\wedge b$ for the minimum, $[a]^+=0\vee a$, $[a]^-=0\vee -a$ , $\lfloor
  a\rfloor$ for the integer part of $a$.  For $f:\R_+\to \R$ let
  $f^\uparrow(t)=\sup_{0\leq s\leq t}f(s)$.

Let $\mathbb{D}=\mathbb{D}([0,\infty),\R)$ be the space of real valued,
  right-continuous functions on $[0,\infty)$ with finite left limits.  We
    endow $\mathbb D$ with the Skorohod $J_1$-topology which makes $\mathbb D$
    a Polish space $\cite{billingsley1968convergence}$.  For $T\geq0$, let
    $\rho_T(x,y)=\sup_{s\in [0,T]}|x(s)-y(s)|$.  Let $e\in \mathbb{D}$ be the
    identity function $e(t)=t$.  For $x\in \mathbb{D}$, let
    $x(t-)=\lim_{s\uparrow t}x(s)$, and let $x^-(t)=x(t-)$ for $t>0$
    and $x^-(0)=x(0)$.

Following Ethier and Kurtz \cite{ethier2009markov} let $\Lambda'$ be the
collection of strictly increasing functions mapping $\R_{+}$ onto $\R_{+}$.
Let $\Lambda\subset \Lambda'$ be the set of Lipschitz continuous functions
such that $\lambda\in \Lambda$ implies $\displaystyle\sup_{s>t\geq
0}\left|\log \frac{\lambda(s)-\lambda(t)}{s-t}\right|<\infty$. 

We will often use \cite{ethier2009markov} Proposition 3.5.3: let
$\{x_n\}\subset \mathbb{D}$ and $x\in \mathbb{D}$.  Then $x_n\xrightarrow{J_1}
x$ if and only if for each $T>0$ there exists $\{\lambda_n\}\subset \Lambda'$
(possibly depending on $T$) such that $\lim_{n\to\infty} \sup_{0\leq t\leq T}
|\lambda_n(t)-t|=0$ and $\lim_{n\to\infty}\sup_{0\leq t\leq T}
|x_n(t)-x(\lambda_n(t))|=0$.

We write $X\sim Y$ if $X$ and $Y$ are equal in distribution.  Weak convergence
of random elements will be denoted by $\Rightarrow$.  We adopt the convention
that a sum of the form $\sum_{i=n}^m$ with $n>m$, or a sum over an empty set
of indices equals zero.

\section{Tandem queue model and main result}
In this section we give a precise description of the tandem queue, specify our assumptions, and state our main result. 

\subsection{Definition of the model}\label{sec:modelNor}
We formulate a model equivalent to the one in Boxma \cite{boxma1979tandem}.
The tandem queueing system consists of two queues Q1 and Q2 in series; both Q1
and Q2 are single-server queues with an unlimited buffer. Jobs enter the
tandem system at Q1. After completion of service at Q1 a job immediately
enters Q2, and when service at Q2, which is the exact same length as
previously experienced in Q1, is completed it leaves the tandem system. Jobs
are served individually and at both servers with the first in first out
discipline.  We assume the system is empty at time zero.

More precisely, at Q1 the {\it exogenous arrival process} $E(\cdot)$ is a
renewal process.  Jump times of this process correspond to times at which jobs
enter the system.  This renewal process is defined from a sequence of
interarrival times $\{u_i\}_{i=1}^\infty$, where $u_1$ denotes the time at
which the first job to arrive after time zero enters the system and $u_i$, $i
\geq 2$, denotes the time between the arrival of the $(i-1)$st and the $i$th
jobs to enter the system after time zero.  Thus, $U_i=\sum_{j=1}^i u_j$ is the
time at which the $i$th arrival enters the system, which is interpreted as
zero if $i = 0$, and $E(t) = \sup\{i \geq 0:U_i\leq t\}$ is the number of
exogenous arrivals by time $t$. We assume that the sequence
$\{u_i\}_{i=1}^\infty$ is an independent and identically distributed sequence
of nonnegative random variables with $\expectation{u_1}=\mu  < \infty$. 

At Q1, the service process, $\{V_i, i = 1, 2, \ldots\}$, is such that $V_i$
records the total amount of service required from the server by the first $i$
arrivals.  More precisely, $\{v_i\}_{i=1}^\infty$ denotes an independent and
identically distributed sequence of strictly positive random variables.  We
interpret $v_i$ as the amount of processing time that the $i$th arrival
requires from both servers.  The $v_i$'s are known as the {\it service times}.
Then, $V_i=\sum_{j=1}^i v_j$, which is taken to be zero if $i = 0$. It is
assumed that $\expectation{v_1}=\nu<\infty$.

For $t\geq 0$, let
\begin{equation*}
I(t)=\sup_{s\leq t}\left[V_{E(s)}-s\right]^-.
\end{equation*}
We interpret $I(t)$ as the cumulative amount of time that the first server has
been idle up to time $t$.  For $n\geq 0$, let
\begin{equation*}
I_n=I(U_n).
\end{equation*}
Then $I_n$ is the cumulative amount of time that the first server has be idle up
to the arrival of the $n$th job in the first queue.

Let $W_i(t)$ denote the (immediate) workload at time $t$ at Q$i$, $i=1,2$,
which is the total amount of time that the server must work in order to
satisfy the remaining service requirement of each job present in the system at
time $t$, ignoring future arrivals.  For $t\geq 0$ we define
\begin{equation*}
W_1(t)=V_{E(t)}-t +I(t).
\end{equation*} 

Let $D_n$ be the {\it transfer time} of the $n$th job.  So, the $n$th job
exits Q1 and enters Q2 at time $D_n$.  Let $d_1=u_1+v_1$ and $d_n=D_n-D_{n-1}$
for $n\geq 2$ be the {\it intertransfer time} between arrivals of the $n-1$st
and $n$th job to the second queue.  For $n\geq 0$ we have
\begin{equation*}
D_n=V_n+I_n.
\end{equation*}

Let $R(t)$ denote the number of transfers to Q2 by time $t$.  For $t\geq 0$ we have
\begin{equation}\label{eq:R}
R(t)=\sup\{n\geq 0:D_n\leq t\}.
\end{equation}

Let $J(t)$ denote the cumulative amount of time that the second server has
been idle up to time $t$, and $W_2(t)$ as the workload in Q2 at time $t$.
That is, for $t\geq 0$ let
\begin{equation*}
\begin{split}
J(t)&=\sup_{s\leq t}\left[V_{R(s)}-s\right]^-,\\
W_2(t)&=V_{R(t)}-t+J(t).
\end{split}
\end{equation*}
If $k$ is the index of the first job in a busy period of the first queue then
$W_1(U_k)=v_k$.  Similarly, $W_2(D_k)=v_k$ if the $k$th job arrives to the
second queue at a time when the second queue is empty.

Finally, let $M_n$ denote the workload in the second queue at the time of the
arrival of the $n$th job to the second queue, which is just the sojourn time
of the $n$th job in the second queue.  Let $M(t)$ be the piecewise constant
right continuous function that agrees with the work load in the second queue
at each transfer time and whose discontinuities are contained in the transfer
times.  We call $M(t)$ the {\it plateau process}.  For integers $n\geq 0$ and
real numbers $t\geq0$ we have
\begin{equation}\label{eq:M_R}
\begin{split}
M_n&=W_2(D_n),\\
M(t)&=M_{R(t)}.
\end{split}
\end{equation}
Finally, we define for  $t\geq 0$, 
\begin{equation}\label{defn:U}
U(t)=U_{\lfloor t \rfloor} \qquad\text{and}\qquad  V(t)=V_{\lfloor t \rfloor}.
\end{equation}

\subsection{Sequence of models, assumptions, and results}\label{sec:modelR} We
now specify a sequence of tandem queueing models indexed by $r\in \R$, where
$r$ increases to $\infty$ through a sequence in $(0,\infty)$.  Each model in
the sequence is defined on the same probability space $(\Omega, \mathcal F,
\mathbb P)$.  The $r$th model in the sequence is defined as in the previous
section where we add a superscript $r$ to each symbol.  In particular, for
$t\geq0$ let $M^r(t)$ denote the plateau process in the $r$th system.

Then $\{v_i^r\}_{i=1}^\infty$ and $\{u_i^r\}_{i=1}^\infty$ are the
service times and interarrival times to the first queue with positive, finite
means $\expectation{v_i^r}=\nu^r$ and $\expectation{u_i^r}=\mu^r$ for each
$i=1,2,\ldots$ independent of each other. Define the following scaled versions
of processes in the $r$th model for a sequence of positive reals
$a_r\to\infty$ and $t\geq0$,
\begin{equation}\label{DiffusionScale}
\begin{array}{lcl}
\bar U^r(t)=r^{-1}U(rt)&  \quad \text{ and }\quad &\bar V^r(t)=r^{-1}V(rt) \\
\check U^r(t)=a_r^{-1}\left(U(rt)-r\mu^r t\right)&  \ \ \  \text{and}\ \ \  &\check V(t)=a_r^{-1}\left(V(rt)-r\nu^r t\right) \\
&{\check M^r(t)=a_r^{-1}M^r(rt)}.
\end{array}
\end{equation}

{\bf Asymptotic assumptions.}  We make the following asymptotic assumptions,
as $r\to\infty$, about our sequence of models.  Assume there is a sequence
$\{a_r\}$ such that $r/a_r\to\infty$, $\check U^r(1)\Rightarrow \mathcal U^*$,
$\check V^r(1)\Rightarrow \mathcal V^*$ in $\R$.  In this case $\mathcal U^*$
and $\mathcal V^*$ are centered infinitely divisible random variables; see
Feller \cite{feller1971introduction} XII.7.  Then we have
$U^r\Rightarrow U^*$ and $V^r\Rightarrow V^*$ in $\mathbb D$, where $U^*$ and
$V^*$ are L\'evy stable motions with $U^*(1)\sim \mathcal U^*$ and
$V^*(1)\sim\mathcal V^*$; see \cite{whitt2002stochastic} supplement 2.4.1.  We
further assume $\lim_{r\to\infty} \mu^r=\lim_{r\to\infty} \nu^r=\mu$ and the
traffic intensity parameter for the $r$th system $\rho^r=\frac{\mu^r}{\nu^r}$
satisfies
\begin{equation*}
\frac{r}{a_r}\left(1-\rho^r\right)\to \gamma\in \R.
\end{equation*}

\begin{defn}\label{defn:F}
Define the mapping $F:\mathbb{D}\times \mathbb{D}\times \R\to \mathbb{D}$ by 
\begin{equation*}
\begin{split}
F(x,y,c)(t)&=\sup_{0\leq s\leq t}\left(y(s)-y(s-)+\sup_{0\leq r\leq s}\left(x(r)-y([r-c]^+\right)\right)\\
&\qquad\qquad-\left.\left.\sup_{0\leq s\leq t}\right(x(s)-y([s-c]^+)\right)
\end{split}
\end{equation*}
\end{defn}
The following is the main result of the paper. 
\begin{thm}\label{Main Theorem} As $r\to\infty$, 
\begin{equation*}
\check M^r\Rightarrow M^*,
\end{equation*}
where $M^*(t)=F(U^*+\gamma \mu e,V^*,0)(t/\mu)$.
\end{thm}

\section{The plateau process as a function of $U$ and
$V$}\label{s.Relationships}
In this section we derive various relationships between the stochastic
processes comprising the tandem queueing model.  These relationships hold for
any of the $r$ indexed models, so we suppress superscripts referring to a
particular model in sequence.

\subsection{The idleness process for the first queue}
This section is a prerequisite for understanding the arrival process in the
second queue.  If the cumulative idleness in the first queue is identically
zero for all time, then the arrival process to the second queue is just a
renewal process formed by the service times.  Here we consider the cumulative
idleness process in the first queue as a discrete time process.  Consider the
model defined in section \ref{sec:modelNor}.

\begin{lem}\label{lem:expressI_n} For each $n\geq 1$,
\begin{equation}\label{eq:expressI_n}
I_n=u_1+\max_{k=1}^{n}\left(\sum_{j=2}^k(u_j-v_{j-1})\right),
\end{equation}
for $n=1,2,\ldots$
\end{lem}
\begin{proof}
We proceed by induction.  First observe that $\sum_{j=2}^1(u_j-v_{j-1})=0$, by convention, so
\begin{equation*}
\max_{k=1}^{n}\left(\sum_{j=2}^k(u_j-v_{j-1})\right)\geq 0
\end{equation*}
 for $n\geq 1$.
$\displaystyle I_1=u_1+\max_{k=1}^1\sum_{j=1}^k(u_j-v_{j-1})=u_1$. For $n=2$,
\begin{equation*}
I_2=u_1+\left[u_2-v_1\right]^+=u_1+\max_{k=1}^{2}\left(\sum_{j=2}^k(u_j-v_{j-1})\right),
\end{equation*}
since there is no additional idleness if the second job arrives while the first job is in service.  This is the base case for the induction.

For the inductive step, assume equation \eqref{eq:expressI_n} holds for $n\geq
2$.  There are two cases.  In the first case the $(n+1)$st job arrives before
the $n$th service is complete.  In this case the first job in the current busy
period had index $i\leq n$, arrived at time $t_i$, and the total amount of
work that has arrived since $t_i$, $\sum_{k=i}^n v_k$ exceeds the amount of
time $\sum_{k=i+1}^{n+1}u_{k}$ since $t_i$.  That is,
\begin{equation*}
\sum_{k=i+1}^{n+1}u_{k}-v_{k-1}< 0,
\end{equation*}
for some $i\leq n$.  Thus
\begin{equation*}
\max_{k=1}^{n+1}\left(\sum_{j=2}^k(u_j-v_{j-1})\right)=\max_{k=1}^{n}\left(\sum_{j=2}^k(u_j-v_{j-1})\right),
\end{equation*}
and the cumulative idle time has not increased 
\begin{equation*}
\displaystyle I_n=I_{n+1}=u_1+\max_{k=1}^{n+1}\left(\sum_{j=2}^k(u_j-v_{j-1})\right).
\end{equation*}

In the second case, the $(n+1)$st job arrives after the $n$th service is
complete, so the total idle time just before the arrival of the $n+1$ job is
$u_1+\sum_{k=2}^{n+1}u_k-v_{k-1}$.  In this case, for any job $i\leq n$, the
total amount of time $\sum_{k=i+1}^{n+1} u_k$ exceeds the total amount of work
$\sum_{k=i}^n v_k$ since $t_i$.  That is,
\begin{equation*}
\sum_{k=i+1}^{n+1}u_{k}-v_{k-1}\geq 0.
\end{equation*}
Thus, 
\begin{equation*}
\left(\sum_{j=2}^k(u_j-v_{j-1})\right)\leq \left(\sum_{j=2}^{n+1}(u_j-v_{j-1})\right)
\end{equation*}
for each $k=2, \ldots, n+1$, and we have $\displaystyle \sum_{j=2}^{n+1}u_j-v_{j-1}=\max_{k=1}^{n+1}\left(\sum_{j=2}^k(u_j-v_{j-1})\right).$
\end{proof}

Note that the departure process of the first queue is equal to the arrival
process $R(\cdot)$ of the second queue.  Since the queueing discipline is
FIFO, the number of jobs that have arrived to the second queue by time $t$ is
the greatest number $N$ such that the total amount of time needed to complete
the first $N$ jobs, $\sum_{k=1}^N v_k$, is less than the amount of time spent
working, $t$ minus the cumulative idle time in the first queue.

\subsection{Workload in the second queue}
In this section we show how to write the plateau process $M(\cdot)$ as a
function of the primitive arrival and service processes.  The following
formula relates sojourn times in the second queue to service times and
idleness in the first queue.  It comes from Lindley recursion
\cite{asmussen1987applied} for a FIFO queue
$W_2(D_{n+1})=v_{n+1}+[W_2(D_n)-d_{n+1}]^+$, where no independence needs to be
assumed about the intertransfer times $d_k$ and service times $v_k$.

\begin{lem}\label{lem:M_nForm}
The sojourn time of the $n^{th}$ job in the second queue is
\begin{equation*}
M_n=\max_{k=1}^n\left\{v_k+I_k\right\}-I_n.
\end{equation*}
\end{lem}
\begin{proof}
Note that the sojourn time of the $n^{th}$ job includes its service time.  The
second queue is initially empty and the service time of the $n$th job is the
same in both queues.  Clearly $I_1=u_1$, since the first queue is empty until
the arrival of the first job.  So,
\begin{equation*}
M_1=v_1=\max_{k=1}^1\left\{v_k+I_k\right\}-I_1.
\end{equation*}
The intertransfer time between the $n$th and $(n+1)$st job is
$d_{n+1}=v_{n+1}+(I_{n+1}-I_n)$.  Proceeding by induction, suppose
$\displaystyle M_n=\max_{k=1}^n\left\{v_k+I_k\right\}-I_n.$  Then, Lindley
recursion gives
\begin{equation*}
\begin{split}
M_{n+1}&=v_{n+1}+\left[M_n-v_{n+1}-\left(I_{n+1}-I_n\right)\right]^+\\
&=v_{n+1}\vee \left( M_n-\left(I_{n+1}-I_n\right)\right)\\
&=v_{n+1}\vee \left(\max_{k=1}^n\left(v_k+I_k\right)-I_n-\left(I_{n+1}-I_n\right)\right)\\
&=\left[\left(v_{n+1}+I_{n+1}\right)\vee \max_{k=1}^n\left(v_k+I_k\right)\right]-I_{n+1}\\
&=\max_{k=1}^{n+1}\left(v_k+I_k\right)-I_{n+1}.
\end{split}
\end{equation*}
\end{proof}

\begin{defn}\label{defn:GH}
Define the translation function $G:\mathbb{D}\times \R\to \mathbb{D}$ by 
\begin{equation*}
G(x,c)(t)=x([t-c]^+),
\end{equation*}
and define $H:\mathbb{D}\times \mathbb{D}\times \mathbb R_{+}\to \mathbb{D}$
as the composition 
\begin{equation*}
H(x,y,c)=\left(x-G(y,c)\right)^\uparrow.
\end{equation*}
More explicitly,
\begin{equation*}
H(x,y,c)(t)=\left.\left.\sup_{0\leq s\leq t}\right(x(s)-y([s-c]^+)\right).
\end{equation*}
\end{defn}

We can write $I_n$ in terms of $V$ and $U$ from \ref{defn:U}.  
\begin{lem}\label{lem:I_n=H} For each $n\geq 1$,
\begin{equation*}
I_n=H(U,V,1)(n),
\end{equation*}
Moreover $H$ is constant on intervals of the form $[n,n+1)$ where $n$ is an
  integer, so for each integer $n$ we have $H(U,V,n)(\lfloor
  t\rfloor)=H(U,V,n)(t)$ for all $t\geq 0$.
\end{lem}
\begin{proof}
The processes $V$ and $U$ are constant between integers so $H$ is constant on
intervals of the form $[n,n+1)$, where $n$ is an integer.  For an integer $k$,
  $v_k=V(k)-V(k-)$ and $u_k=U(k)-U(k-)$.  By lemma \ref{lem:expressI_n},
\begin{equation*}
\begin{split}
I_{n}&=u_1+\max_{k=1}^{n}\left(\sum_{j=2}^k(u_j-v_{j-1})\right)\\
&=u_1 +\max_{k=1}^{n}\left(\sum_{j=2}^ku_j-\sum_{j=1}^{k-1}v_j\right)\\
&=\max_{k=1}^{n}\left(\sum_{j=1}^ku_j-\sum_{j=1}^{k-1}v_j\right)\\
&=\max_{k=1}^{n}\left(U(k)-V(k-1)\right)\\
&=\sup_{0\leq s\leq n}\left(U(s)-V([s-1]^+)\right)\\
&=\sup_{0\leq s\leq n}\left(U(s)-G(V,1)(s)\right)\\
&=H(U,V,1)(n).
\end{split}
\end{equation*}
\end{proof}

Now we can write $R$ in terms of $U$ and $V$.
\begin{coro}\label{coro:RwithH}
\begin{equation*}
R(t)=\max\left\{m\geq 0:V(m)+H(U,V,1)(m)\leq t\right\}.
\end{equation*}
\end{coro}
\begin{proof}
From Definition \eqref{eq:R} we have $R(t)=\max\{N\geq 0: \sum_{k=1}^N
v_k+I_N\leq t\}$.  We have $\sum_{k=1}^N v_k=V(N)$ by Definition
\ref{defn:U} and $I_N=H(U,V,1)(N)$ by Lemma \ref{lem:I_n=H}.
\end{proof}

We can now write the plateau process in terms of the function $F$ defined in
section \ref{sec:modelR}.  By Definitions \ref{defn:F} and \ref{defn:GH},
\begin{equation*}
F(x,y,c)=\left(y-y^-+H(x,y,c)\right)^\uparrow-H(x,y,c),
\end{equation*}
or more explicitly,
\begin{equation*}
F(x,y,c)(t)=\sup_{0\leq s\leq t}\left(y(s)-y(s-)+H(x,y,c)(s)\right)-H(x,y,c)(t).
\end{equation*}

\begin{lem}\label{lem:M=F}  For all $t\geq 0$,
\begin{equation*}
M_{\lfloor t\rfloor}=F(U,V,1)(t).
\end{equation*}
\end{lem}
\begin{proof}
By lemma \ref{lem:M_nForm}
\begin{equation*}
\begin{split}
M_{\lfloor t\rfloor}&=\max_{k=1}^{\lfloor t\rfloor}\left(v_k+I_k\right)-I_{\lfloor t\rfloor}\\
&=\max_{k=1}^{\lfloor t\rfloor}\left(V(k)-V(k-)+I_k\right)-I_{\lfloor t\rfloor}\\
&=\max_{k=1}^{\lfloor t\rfloor}\left(V(k)-V(k-)+H(U,V,1)(k)\right)-H(U,V,1)(\lfloor t\rfloor)
\end{split}
\end{equation*}
by lemma \ref{lem:I_n=H}.  For a positive integer $k$ we have $H(U,V,1)(t)$ is
constant for $t$ in $[k,k+1)$ and $V(k)-V(k-)\geq V(t)-V(t-)$ for $t$ in
  $[k,k+1)$.  Thus, $V(t)-V(t-)+H(U,V,1)(t)$ is maximized when $t$ is an
    integer.  Thus,
\begin{equation*}
\begin{split}
M_{\lfloor t\rfloor}&=\sup_{0\leq s\leq t}\left(V(s)-V(s-)+H(U,V,1)(s)\right)-H(U,V,1)(t)\\
&=F(U,V,1)(t).
\end{split}
\end{equation*}
\end{proof}

Finally we can express $M(\cdot)$ as function of $U$ and $V$.  By Definition
\eqref{eq:M_R}, $M(t)$ is the composition $M_{(\cdot)}$ with the arrival
process to the second queue.  That is,
\begin{equation*}
\begin{split}
M(t)&=M_{R(t)}\\
&=F(U,V,1)(\max\left\{m\geq 0:V(m)+H(U,V,1)(m)\leq t\right\}).
\end{split}
\end{equation*}
Notice that the plateau process is greater than or equal to the workload in
the second queue at each time, that is $M(t)\geq W_2(t)$ for each $t\geq 0$.

\section{Continuity properties of $G,H,$ and $F$}\label{s.Continuity} 
Note that the function $F$ is not continuous everywhere. For example, let
$x_n= x= 1_{[1,\infty)}+1_{[2,\infty)}$, let $y=
  1_{[1,\infty)}$, and let $y_n= y(\cdot-1/n)$ so that $(x_n,y_n,0)$
    clearly converges to $(x,y,0)$ in
    $\mathbb{D}\times\mathbb{D}\times\mathbb{R}$. Then $F(x_n,y_n,0)=
    y_n$ which converges in the Skorohod $J_1$-topology to $y$. But this does
    not equal $F(x,y,0)=1_{[1,2)}$, so $F$ is not continuous at $(x,y,0)$. 

In this section we identify a subset of the domain of $F$ that almost surely
contains the limits of the processes we are interested in and on which $F$ is
indeed continuous.  This result is obtained by treating $F$ as a composition of
continuous functions.  The strategy of proof is similar to showing addition is
continuous on a large subset of $\mathbb{D}\times\mathbb{D}$ (see e.g.\
\cite{whitt1980some}).

\begin{lem}\label{GisContinuous}
For any $x\in \mathbb{D}$, $G$ is continuous at $(x,0)$ in the product topology on $\mathbb D\times \R$.
\end{lem}

\begin{proof}
Let $c_n$ be a sequence in $\R$ with $c_n\to 0$, and let $x_n\to x$ in
$\mathbb D$.  Then for each $T>0$ there exists $\{\lambda_n\}\subset \Lambda$
such that $\sup_{0\leq t\leq T}|\lambda_n(t)-t|\to0$ as $n\to\infty$ and
$\sup_{0\leq t\leq T}|x_n(t)-x(\lambda_n(t))|\to0$ as $n\to\infty$.

For each $n=1,2,\ldots$ define 
$$\tilde\lambda_n(t)=\left\{\begin{array}{ll} 
\lambda_n(t-c_n), &\text{ if } t\geq 2|c_n|,\\
\lambda_n\left(\left(1-\frac{\sgn(c_n)}{2}\right)t\right), &\text{ if } t< 2|c_n|,
\end{array}\right. $$
where $\sgn(c_n)=-1$ if $c_n<0$, $\sgn(c_n)=1$ if $c_n> 0$, and $\sgn(c_n)=0$ if $c_n=0$.

We have $\{\tilde\lambda_n\}\subset \Lambda$ because each $\tilde\lambda_n$ is
the composition of two functions in $\Lambda$.  Now,

\begin{multline*}
\sup_{0\leq t\leq T}\left|\tilde\lambda_n(t)-t\right|=\left(\sup_{0\leq t< 2|c_n|}\left|\tilde\lambda_n(t)-t\right|\right)\vee\left(\sup_{2|c_n|\leq t\leq T}\left|\tilde\lambda_n(t)-t\right|\right)\\
=\left(\sup_{0\leq t<2|c_n|}\left|\lambda_n\left(\left(1-\frac{\sgn(c_n)}{2}\right)t\right) -t\right|\right)\vee\left(\sup_{2|c_n|\leq t\leq T}\left|\lambda_n(t-c_n)-t\right|\right)\\
\leq\left(\sup_{0\leq t<2|c_n|}\left|\lambda_n\left(\left(1-\frac{\sgn(c_n)}{2}\right)t\right) -\left(1-\frac{\sgn(c_n)}{2}\right)t\right|\right.\\
\left.+\sup_{0\leq t\leq 2|c_n|}\left|\left(1-\frac{\sgn(c_n)}{2}\right)t-t\right|\right)\vee\left(\sup_{2|c_n|\leq t\leq T}\left|\lambda_n(t-c_n)-(t-c_n)\right|+|c_n|\right).
\end{multline*}

When $0\leq t<2|c_n|$ we have $0\leq \left(1-\frac{\sgn(c_n)}{2}\right)t\leq 3|c_n|$, so
\begin{equation*}\begin{split}
\sup_{0\leq t\leq T}\left|\tilde\lambda_n(t)-t\right|&\leq\left(\sup_{0\leq t<3|c_n|}\left|\lambda_n\left(t\right) -t\right|+3|c_n|\right)\\
&\qquad\vee\left(\sup_{2|c_n|-c_n\leq t\leq T-c_n}\left|\lambda_n(t)-t\right|+|c_n|\right)\\
&\leq\sup_{0\leq t\leq T}\left|\lambda_n(t)-t\right|+3|c_n|,
\end{split}\end{equation*}
so $\sup_{0\leq t\leq T}\left|\tilde\lambda_n(t)-t\right|\to 0$ as $n\to\infty$.

Now, it suffices to show $\displaystyle \sup_{0\leq t\leq T}\left|G(x_n,c_n)(t)-G(x,0)(\tilde\lambda_n(t))\right|\to 0$ by \cite{ethier2009markov} Proposition 3.5.3.  We have

\begin{multline}\label{eq:2|c_n|<t}
\sup_{2|c_n|\leq t\leq T}\left|G(x_n,c_n)(t)-G(x,0)(\tilde\lambda_n(t))\right|\\
=\sup_{2|c_n|\leq t\leq T}\left|x_n([t-c_n]^+)-x(\tilde\lambda_n(t))\right|\\
=\sup_{2|c_n|\leq t\leq T}\left|x_n( t-c_n)-x(\lambda_n(t-c_n))\right|\\
=\sup_{2|c_n|-c_n\leq t\leq T-c_n}\left|x_n(t)-x(\lambda_n(t))\right|\to 0
\end{multline}
So it suffices to show $\sup_{0\leq t<
2|c_n|}\left|G(x_n,c_n)(t)-G(x,0)(\tilde\lambda_n(t))\right|\to 0$.

Fix $\eps>0$ and let $\eta>0$ such that $\sup_{0\leq t\leq
\eta}|x(0)-x(t)|<\eps$ by right continuity of $x$ at zero.  Now, for $n$ so
large that $|c_n|<\min(T/3,\eta/6)$, $\sup_{0\leq t\leq
T}|\lambda_n(t)-t|<\eps \wedge \eta/2$, and $\sup_{0\leq t\leq
T}|x_n(t)-x(\lambda_n(t))|<\eps$ consider the $c_n<0$, $c_n>0$, and $c_n=0$
cases.

If $c_n<0$,
\begin{multline*}
\sup_{0\leq t< 2|c_n|}\left|G(x_n,c_n)(t)-G(x,0)(\tilde\lambda_n(t))\right|\\
=\sup_{0\leq t< 2|c_n|}\left|x_n([t-c_n]^+)-x(\tilde\lambda_n(t))\right|\\
=\sup_{0\leq t< -2c_n}\left|x_n(t-c_n)-x(\lambda_n(3t/2))\right|\\
\leq\sup_{0\leq t<-2c_n}\left|x_n(t-c_n)-x(\lambda_n(t-c_n))\right|+\left|x(\lambda_n(t-c_n))-x(\lambda_n(3t/2))\right|\\
\leq \sup_{0\leq t\leq T}\left|x_n(t)-x(\lambda_n(t))\right|+\sup_{0\leq t< -2c_n}\left|x(\lambda_n(t-c_n))-x(\lambda_n(3t/2))\right|\\
\leq \sup_{0\leq t\leq T}\left|x_n(t)-x(\lambda_n(t))\right|+\sup_{0\leq t< -2c_n}\left|x(\lambda_n(t-c_n))\right|+\sup_{0\leq t<-2c_n}\left|x(\lambda_n(3t/2))\right|.
\end{multline*}

We have $(t-c_n)\vee (3t/2)\leq -3c_n$ for $0\leq t<-2c_n$, and so 
\begin{equation*}
\lambda_n(t-c_n)\vee \lambda_n(3t/2)\leq \lambda_n(-3c_n)\leq -3c_n+\eta/2\leq \eta.
\end{equation*}

Thus,
\begin{multline*}
\sup_{0\leq t< 2|c_n|}\left|G(x_n,c_n)(t)-G(x,0)(\tilde\lambda_n(t))\right|\\
\leq \eps+\sup_{0\leq t< -2c_n}\left|x(\lambda_n(t-c_n))\right|+\sup_{0\leq t<-2c_n}\left|x(\lambda_n(3t/2))\right|\\
\leq \eps+\sup_{0\leq t\leq \eta}\left|x(t)\right|+\sup_{0\leq t\leq\eta}\left|x(t)\right|\leq 3\eps
\end{multline*}

If $c_n>0$,

\begin{multline}\label{c_n>0}
\sup_{0\leq t< 2|c_n|}\left|G(x_n,c_n)(t)-G(x,0)(\tilde\lambda_n(t))\right|\\
=\sup_{0\leq t< 2c_n}\left|x_n([t-c_n]^+)-x(\tilde\lambda_n(t))\right|\\
=\sup_{0\leq t< 2c_n}\left|x_n([t-c_n]^+)-x(\lambda_n(t/2))\right|\\
\leq\sup_{0\leq t< c_n}\left|x_n(0)-x(\lambda_n(t/2))\right|\vee\sup_{c_n\leq t< 2c_n}\left|x_n(t-c_n)-x(\lambda_n(t/2))\right|.
\end{multline}
For the first term,
\begin{multline*}
\sup_{0\leq t\leq c_n}\left|x_n(0)-x(\lambda_n(t/2))\right|\leq \sup_{0\leq t< c_n}\left|x_n(0)-x(0)|+|x(0)-x(\lambda_n(t/2))\right|\\
= \left|x_n(0)-x(\lambda_n(0))\right| + \sup_{0\leq t<c_n}\left|x(0)-x(\lambda_n(t/2))\right|\\
\leq \sup_{0\leq t\leq T} \left|x_n(t)-x(\lambda_n(t))\right|+\sup_{0\leq t\leq \eta}\left|x(0)-x(t)\right|\leq 2\eps,
\end{multline*}
since $\lambda_n(t/2)\leq \lambda_n(c_n/2)\leq c_n/2+\eta/2\leq \eta$ for $0\leq t\leq c_n$.
For the second term,
\begin{multline*}
\sup_{c_n\leq t< 2c_n}\left|x_n(t-c_n)-x(\lambda_n(t/2))\right|
=\sup_{0\leq t< c_n}\left|x_n(t)-x\left(\lambda_n\left(\frac{t+c_n}{2}\right)\right)\right|\\
\leq \sup_{0\leq t< c_n}\left|x_n(t)-x(\lambda_n(t))\right|+\left|x(\lambda_n(t))-x\left(\lambda_n\left(\frac{t+c_n}{2}\right)\right)\right|\\
\leq \eps+\sup_{0\leq t< c_n}\left|x(\lambda_n(t))-x(0)+x(0)-x\left(\lambda_n\left(\frac{t+c_n}{2}\right)\right)\right|\\
\leq \eps+\sup_{0\leq t< c_n}\left|x(\lambda_n(t))-x(0)\right|+\sup_{0\leq t< c_n}\left|x(0)-x\left(\lambda_n\left(\frac{t+c_n}{2}\right)\right)\right|\\
\leq \eps+2\sup_{0\leq t< \eta}\left|x(0)-x(t)\right|\leq 3\eps,\\
\end{multline*}
since $\lambda_n(t)\vee\lambda_n(\frac{t+c_n}{2})\leq \lambda_n(c_n)\leq c_n+\eta/2\leq \eta$ for $0\leq t\leq c_n$.

If $c_n=0$ then $\tilde\lambda_n=\lambda_n$ so $G(x_n,c_n)(t)-G(x,0)(\tilde\lambda_n(t))=x_n(t)-x(\lambda_n(t))$, which converges to zero uniformly by assumption.

So in all three cases we have 
\begin{equation*}
\sup_{0\leq t< 2|c_n|}\left|G(x_n,c_n)(t)-G(x,0)(\tilde\lambda_n(t))\right|\leq 3\eps.
\end{equation*}
Together with \eqref{eq:2|c_n|<t} and since $\epsilon$ was arbitrary, we have
\begin{equation*}
\sup_{0\leq t\leq T}\left|G(x_n,c_n)(t)-G(x,0)(\tilde\lambda_n(t))\right|\to 0
\end{equation*}
as $n\to\infty$.

So we have $G(x_n,c_n)\to G(x,0)$ on $\mathbb D$.
\end{proof}

For $x\in \mathbb D$, let Disc$(x)$ denote the set of discontinuities of $x$.

\begin{lem}\label{HIsContinuous}
$H$ is continuous at $(x,y,0)$ for all $x,y\in \mathbb{D}$ such that 
\begin{equation*}
\text{Disc}(x)\cap\text{Disc}(y)=\varnothing.
\end{equation*}
\end{lem}
\begin{proof}
Let $c_n\in \R$ with $c_n\to 0$ and let $x_n$ and $y_n$ be in $\mathbb{D}$
such that $x_n\to x$ and $y_n\to y$ and fix a time $T>0$.  Let $z_n=y_n-x_n$
and $z=y-x$.  Since Disc$(x)\cap$Disc$(-y)=\varnothing$, \cite{whitt1980some}
Theorem 4.1 tells us that there exists $\{\lambda_n\}\subset \Lambda'$ such
that $\rho_T(\lambda_n,e)\to 0$ and $\rho_T(z_n,z\circ \lambda_n)\to0$.  Since
$G$ is continuous at $(z,0)$ by lemma \ref{GisContinuous}, and $(z_n,c_n)\to
(z,0)$ we have $\{\tilde\lambda_n\}\subset \Lambda'$ such that
$\rho_T(\tilde\lambda_n, e)\to 0$ and
$\rho_T(G(z_n,c_n),z\circ\tilde\lambda_n)\to0$.  In fact, we may construct
$\tilde\lambda_n$ as in the proof of $\ref{GisContinuous}$.  Since $x\mapsto
x^\uparrow$ is continuous on $\mathbb{D}$ and
$(x)^\uparrow\circ\tilde\lambda=(x\circ\tilde\lambda)^\uparrow$, we have
$\rho_T(H(x_n,y_n,c_n),H(x,y,0)\circ\tilde\lambda_n)\to 0$.  Since $T$ was
arbitrary we have $H$ is continuous $(x,y,0)$. 
\end{proof}

\begin{lem}\label{JumpsofH} For all $x,y\in\mathbb D$,
\begin{equation*}
\text{Disc}(H(x,y,0))\subset\{t:y(t)-y(t-)>0\}\cup\{t:x(t)-x(t-)<0\}.
\end{equation*}
In particular, if $\{t:x(t)-x(t-)<0\}=\varnothing$, then
\begin{equation*}
\text{Disc}(H(x,y,0))\subset\text{Disc}(y).
\end{equation*}
\end{lem}
\begin{proof}
$\text{Disc}(H(x,y,0))=\{t:H(x,y,0)(t)-H(x,y,0)(t-)\neq 0\}=\{t:H(x,y,0)(t)-H(x,y,0)(t-)> 0\}$ since $H(x,y,0)$ is nondecreasing.  Thus,
\begin{equation*}
\begin{split}
\text{Disc}(H(x,y,0))&\subset \{t: (y-x)(t)-(y-x)(t-)>0\}\\
&\subset\{t:y(t)-y(t-)>0\}\cup\{t:x(t)-x(t-)<0\}.
\end{split}
\end{equation*}
\end{proof}

\begin{lem}\label{Realign} 
Let $\lambda_n$ and $\gamma_n$ be strictly increasing homeomorphisms from
$[0,T]$ onto $[0,T]$ and $x_n,x\in \mathbb{D}$ such that for some finite
collection $\{t_j\}_{j=0}^N\subset [0,T]$ with
\renewcommand{\labelenumi}{{\rm(}$\roman{enumi}${\rm)}}
\begin{enumerate}
\item $0=t_0<t_1<\cdots<t_N=T$ we have $\lambda_n^{-1}(t_j)=\gamma_n^{-1}(t_j)$ for each $j=0,1,2,\ldots, N$,
\item $\rho_T(x_n, x\circ \lambda_n)<\eps$, and
\item $w(x,[t_{j-1},t_j))=\sup\left(|x(t)-x(s)|:t,s\in [t_{j-1},t_j)\right)<\eps$ for each $j=1,2,\ldots, N$,
\end{enumerate}
then 
\begin{equation*}
\rho_T(x_n, x\circ \gamma_n)<3\eps.
\end{equation*}
\end{lem}
\begin{proof}
Let $r_j=\gamma_n^{-1}(t_j)=\lambda_n^{-1}(t_j)$ for $j=0,1,\ldots, N$, so that $\cup_{j=1}^N \left[r_{j-1},r_j\right)=\cup_{j=1}^N \left[t_{j-1},t_j\right)=[0,T)$.
Then 
\begin{equation*}\begin{split}
\rho_T(x_n,x\circ \gamma_n)&=\sup_{0\leq t\leq T}\left|x_n(t)-x(\gamma_n(t))\right|\\
&=\max_{k=1}^N\sup_{r_{j-1}\leq t< r_j}\left|x_n(t)-x(\gamma_n(t))\right|\vee|x_n(T)-x(T)|\\
&=\max_{k=1}^N\sup_{t_{j-1}\leq t< t_j}\left|x_n(\gamma_n^{-1}(t))-x(t)\right|\vee|x_n(T)-x(T)|\\
&=\max_{k=1}^N\sup_{t_{j-1}\leq t< t_j}\left|x_n(\gamma_n^{-1}(t))-x(t_{j-1})+x(t_{j-1})-x(t)\right|\\
&\qquad\vee|x_n(T)-x(T)|,
\end{split}\end{equation*}
and so
\begin{equation*}\begin{split}
\rho_T(x_n,x\circ \gamma_n)&\leq\max_{k=1}^N\left(\sup_{t_{j-1}\leq t< t_j}\left|x_n(\gamma_n^{-1}(t))-x(t_{j-1})\right|+w(x,[t_{j-1},t_j))\right)\\
&\qquad\vee|x_n(T)-x(T)|\\
&\leq\max_{k=1}^N\left(\sup_{r_{j-1}\leq t< r_j}\left|x_n(t)-x(\lambda_n(r_{j-1}))\right|+\eps \right)\\
&\qquad\vee|x_n(T)-x(T)|\\
&\leq\max_{k=1}^N\left(\sup_{r_{j-1}\leq t< r_j}\left|x_n(t)-x(\lambda_n(t))\right|\right.\\
&\qquad+\left.{\vphantom{\sup_{r_{j-1}\leq t< r_j}\left|x_n(t)-x(\lambda_n(t))\right|}}\left|x(\lambda_n(t))-x(\lambda_n(r_{j-1}))\right|+\eps \right)\vee|x_n(T)-x(T)|\\
&\leq\max_{k=1}^N\left(\sup_{r_{j-1}\leq t< r_j}\left|x_n(t)-x(\lambda_n(t))\right|+w(x,[t_{j-1},t_j))+\eps \right)\\
&\qquad\vee|x_n(T)-x(T)|\\
&\leq\max_{k=1}^N\left(\sup_{r_{j-1}\leq t< r_j}\left|x_n(t)-x(\lambda_n(t))\right|+2\eps \right)\vee|x_n(T)-x(T)|\\
&\leq \rho_T(x_n,x\circ \lambda_n)+2\eps\\
&\leq 3\eps.
\end{split}\end{equation*}
\end{proof}

Finally, we prove that $F$ is continuous on a relevant set.
\begin{lem}\label{FisContinuous}
$F$ is continuous at $(x,y,0)$ in the product topology on $\mathbb{D}\times \mathbb{D}\times \R$, for all $x$ and $y\in \mathbb{D}$ with $\text{Disc}(x)\cap\text{Disc}(y)=\varnothing$ and
\begin{equation*}
\{t:y(t)-y(t-)<0\}=\varnothing.
\end{equation*}
\end{lem}
\begin{proof}
Let $T>0$, let $\rho_T$ be the uniform metric on function from $[0,T]$ to
$\R$, and fix $\eps>0$.  Apply Lemma 1 on page 110 of
\cite{billingsley1968convergence}  to construct finite subsets $A_1=\{t_j'\}$
and $A_2=\{s_j\}$ of $[0,T]$ such that $0=t_0'<\cdots<t_k'=T$,
$0=s_0<\cdots<s_m=T$,
$w(y;[t_{j-1}',t_j'))=\sup\{|y(s)-y(t)|:s,t\in[t_{j-1}',t_j')\}<\eps$ and
  $w(H(x,y,0);[s_{j-1},s_j))<\eps$ for all $j$.  Since
    Disc$(y)\cap$Disc$(H(x,y,0))\subset$ Disc$(x)\cap$Disc$(y)=\varnothing$,
    the two sets $A_1$ and $A_2$ can be chosen so that $A_1\cap A_2=\{0,T\}$.
    Note that $w(y;[t_{j-1},t_{j}))<\eps$ and $w(H(x,y,0);[t_{j-1},t_j))<\eps$
      for $\{t_j\}=A_1\cup A_2$.  Let $2\delta$ be the distance between the
      closest two points in $A_1\cup A_2$.  Choose $n_0$ and homeomorphisms
      $\lambda_n$ and $\mu_n$ in $\Lambda$ so that
      \renewcommand{\labelenumi}{{\rm(}$\roman{enumi}${\rm)}}
\begin{enumerate}
\item $\rho_T(y_n, y\circ \lambda_n)<(\delta\wedge \eps)$,
\item $\rho_T(\lambda_n, e)<(\delta\wedge \eps)$,
\item $\rho_T(H(x_n,y_n,c_n), H(x,y,0)\circ \mu_n)<(\delta\wedge \eps)$, and
\item $\rho_T(\mu_n, e)<(\delta\wedge \eps)$
\end{enumerate}
for $n\geq n_0$.  Thus for $n\geq n_0$ $$\lambda_n^{-1}(A_1)\cap
\mu_n^{-1}(A_2)=\{0,T\}$$ and $\{r_j\}=\lambda_n^{-1}(A_1)\cup
\mu_n^{-1}(A_2)$ has corresponding points in the same order as
$\{t_j\}=A_1\cup A_2$.  Let $\gamma_n$ be homeomorphisms of $[0,T]$ defined by
$$\gamma_n(r_j)=t_j$$ for corresponding points $r_j\in \lambda_n^{-1}(A_1)\cup
\mu_n^{-1}(A_2)$ and $t_j\in A_1\cup A_2$ and by linear interpolation
elsewhere.

Note that for each $r_j\in \lambda_n^{-1}(A_1)\cup \mu_n^{-1}(A_2)$ either
\begin{equation*}
\lambda_n(r_j)=t_j\qquad \text{ or }\qquad \mu_n(r_j)=t_j.
\end{equation*}
Since $t\mapsto \left|\gamma_n(t)-t\right|$ is continuous the maximum is
attained at some critical point (exposed point) $r_j$, so
$\rho_T(\gamma_n,e)<\rho_T(\lambda_n,e)\vee\rho_T(\mu_n, e)<\eps$.  Now,
\begin{equation*}
\begin{split}\label{whole}
\rho_T(F(x_n,y_n,&c_n), F(x,y,0)\circ\gamma_n)\\
&\leq \rho_T\left(\left(y_n-y_n^-+H(x_n,y_n,c_n)\right)^\uparrow,\left(\left(y-y^-+H(x,y,0)\right)^\uparrow\right)\circ \gamma_n\right)\\
&\qquad+\rho_T\left(H(x_n,y_n,c_n),\left(H(x,y,0)\right)\circ \gamma_n\right).
\end{split}
\end{equation*}
For the first term we have
\begin{equation}\label{eq:whole1}
\begin{split}
&\rho_T\left(\left(y_n-y_n^-+H(x_n,y_n,c_n)\right)^\uparrow,\left(\left(y-y^-+H(x,y,0)\right)^\uparrow\right)\circ \gamma_n\right)\\
&\leq \rho_T\left(y_n-y_n^-+H(x_n,y_n,c_n),\left(y-y^-+H(x,y,0)\right)\circ \gamma_n\right),
\end{split}
\end{equation}
and
\begin{multline}\label{eq:whole2}
\rho_T\left(y_n-y_n^-+H(x_n,y_n,c_n),\left(y-y^-+H(x,y,0)\right)\circ \gamma_n\right)\\
\leq  \rho_T\left(y_n,y\circ\gamma_n\right)+\rho_T\left(y_n^-,y^-\circ\gamma_n\right)+\rho_T\left(H(x_n,y_n,c_n),H(x,y,0)\circ \gamma_n\right).
\end{multline}
Since $\gamma_n$ is strictly increasing,
\begin{equation*}
\begin{split}
\rho_T\left(y_n^-,y^-\circ\gamma_n\right)&=\sup_{0\leq t\leq T}\left|\lim_{s\nearrow t}y_n(s)-\lim_{r\nearrow \gamma_n(t)}y(r)\right|\\
&=\sup_{0\leq t\leq T}\left|\lim_{s\nearrow t}y_n(s)-\lim_{r\nearrow t}y(\gamma_n(r))\right|,
\end{split}
\end{equation*}
and so
\begin{equation*}
\rho_T\left(y_n^-,y^-\circ\gamma_n\right)\leq\sup_{0\leq t\leq T}\left|y_n(t)-y(\gamma_n(t))\right|,
\end{equation*}
since the left limit of $y_n$ and $y\circ \gamma_n$ exist at each $t$.  Therefore,
\begin{equation}\label{eq:whole3}
\rho_T\left(y_n^-,y^-\circ\gamma_n\right)\leq\rho_T\left(y_n,y\circ\gamma_n\right)
\end{equation}
Combining (\ref{whole},\ref{eq:whole1},\ref{eq:whole2},\ref{eq:whole3}) we have,
\begin{equation*}
\begin{split}
&\rho_T(F(x_n,y_n,c_n), F(x,y,0)\circ\gamma_n)\\
&\leq\rho_T\left(\left(y_n-y_n^-+H(x_n,y_n,c_n)\right)^\uparrow,\left(\left(y-y^-+H(x,y,0)\right)^\uparrow\right)\circ \gamma_n\right)\\
&\qquad +\rho_T\left(H(x_n,y_n,c_n),H(x,y,0)\circ \gamma_n\right)\\
&\leq  2\rho_T\left(y_n,y\circ\gamma_n\right)+2\rho_T\left(H(x_n,y_n,c_n),H(x,y,0)\circ \gamma_n\right)\\
&\leq 12\eps,
\end{split}
\end{equation*}
by lemma $\ref{Realign}$.
\end{proof}

\section{Scaling limit of the plateau process}
In this section we prove several results concerning the sequence of models,
and then combine these to prove Theorem \ref{Main Theorem}.  We begin by
showing that the function $H$ scales nicely when no centering is required.

\begin{lem}\label{HScale}
For positive constants $a_n$ and $n$,
\begin{equation*}
a_n^{-1}H(x,y,c)(nt)=H(x^n,y^n,c/n)(t),
\end{equation*}
for all $t\geq0$, where $x^n(t)=a_n^{-1}x(nt)$ and $y^n(t)=a_n^{-1}y(nt).$
\end{lem}
\begin{proof}  By definition,
\begin{equation*}
\begin{split}
a_n^{-1}H(x,y,c)(nt)&=a_n^{-1}\sup_{0\leq s\leq nt}\left(x(s)-y(\left[s-c\right]^+)\right)\\
&=\sup_{0\leq s\leq t}\left(a_n^{-1}x(ns)-a_n^{-1}y(\left[ns-c\right]^+)\right)\\
&=\sup_{0\leq s\leq t}\left(a_n^{-1}x(ns)-a_n^{-1}y(n\left[s-c/n\right]^+)\right)\\
&=\sup_{0\leq s\leq t}\left(x^n(s)-y^n(\left[s-c/n\right]^+)\right)\\
&=H(x^n,y^n,c/n)(t)
\end{split}
\end{equation*}
\end{proof}

\begin{lem}\label{NoNegativeJumps}
The set $\mathscr K=\{x\in \mathbb{D}:x(t)-x(t-)\geq0 \text{ for each }t\in (0,\infty)\}$ is closed in $\mathbb{D}$.
\end{lem}
\begin{proof}
Let $\{x_n\}$ be a sequence in $\mathscr{K}$ such that $x_n\to x$.  Fix
$t_0\in (0,\infty)$ with $x(t_0)-x(t_0-)\neq0$.  There exists $t_n\to t_0$
with $x_n(t_n)-x_n(t_n-)\to x(t_0)-x(t_0-)$ by \cite{jacod1987limit}
proposition VI.2.1.  We have $x_n(t_n)-x_n(t_n-)\geq 0$ for each $n$ since
$x_n\in \mathscr{K}$, so $x(t_0)-x(t_0-)\geq 0$ and we must have $x\in
\mathscr K$.
\end{proof}

The next Lemma establishes a joint convergence involving the primitive input
processes.  Recall that $\check U^r\Rightarrow U^*$ and $\check V^r\Rightarrow
V^*$ in $\mathbb D$.
\begin{lem}\label{JointConvergence}
For any sequence of real numbers $c_r\to c$,
\begin{equation*}
(\check U^r+c_re,\check V^r,1/r)\Rightarrow (U^*+ce,V^*,0),
\end{equation*}
in the product topology on $\mathbb{D}\times \mathbb{D}\times \R$.  Moreover,
\begin{equation*}
\text{Disc}(U^*+ce)\cap\text{Disc}(V^*)=\varnothing \text{ a.s.}
\end{equation*}
 and $\{t:V^*(t)-V^*(t-)<0\}=\varnothing \text{a.s.}$
\end{lem}
\begin{proof}
Since $ce$ is continuous, $\check U^r\Rightarrow U^*$, and $c_re\Rightarrow
ce$ we have  $\check U^r+c_re\Rightarrow U^*+ce$ by \cite{whitt1980some}.  We
have joint convergence $(\check U^r+c_re,\check V^r)\Rightarrow (U^*+ce,V^*)$
since $\check V^r$ is independent of $\check U^r$ and therefore $\check
U^r+c_re$ is independent of $\check V^r$ because $c_r$ is constant in
$\omega$, \cite{whitt2002stochastic} Theorem 11.4.4, moreover $U^*$ is
independent of $V^*$.  Since $1/r$ is constant in $\omega$ we have $1/r\to 0$
in probability so \cite{billingsley1968convergence} Theorem 4.4 gives joint
convergence
\begin{equation*}
(\check V^r+c_re,\check U^r,1/n)\Rightarrow (U^*+ce,V^*,0).
\end{equation*}
$V^*$ is a stable L\'evy motion by 2.4.1 of the online supplement to
\cite{whitt2002stochastic}.  So $V^*$ has no fixed discontinuities:
$\probability{U^*(t)=U^*(t-)}=1$ for all $t\in (0,\infty)$.   By
\cite{whitt1980some} Lemma 4.3, gives
$\probability{\text{Disc}(U^*)\cap\text{Disc}(V^*)=\varnothing }=1$ and since
$ce$ is continuous we have
\begin{equation*}
\probability{\text{Disc}(U^*+ce)\cap\text{Disc}(V^*)=\varnothing }=1.
\end{equation*}
Finally, $\probability{\check V^r\in \mathscr{K}}=1$, $\check V^r\Rightarrow V^*$, and $\mathscr{K}$ is closed by Lemma \ref{NoNegativeJumps}, so the Portmanteau theorem gives 
\begin{equation*}
\probability{V^*\in \mathscr K}\geq\limsup_{n\to\infty}\probability{\check V^r\in \mathscr{K}}=1.
\end{equation*}
\end{proof}

For each $r>0$ and $t\geq 0$ define $\bar D^r(t)=\frac{1}{r}D^r(rt)$.  Using
Corollary \ref{coro:RwithH} under this fluid scaling, we have for all $t\geq
0$,
\begin{equation*}
\bar R^r(t)=\frac{1}{r}R(rt).
\end{equation*}
We will need the fluid limit of $\bar D^r(\cdot)$.
\begin{lem}\label{lem:RnScale}
As $r\to\infty$,
\begin{equation*}
\bar R^r\Rightarrow e/\mu
\end{equation*}
\end{lem}
\begin{proof}
$\check U^r(1)\Rightarrow U^*(1)$ implies $\frac{r}{a_r}\left(\bar U^r(1)-\mu^r\right)\Rightarrow U^*(1)$, but $r/a_r\to\infty$ implies $\bar U^r(1)-\mu_r\Rightarrow 0$.  Since $\mu^r\to \mu$ we have $\bar U^r(1) \Rightarrow \mu$.  By Theorem 2.4.1 of the internet supplement to \cite{whitt2002stochastic}, we have $\bar U^r\Rightarrow \mu e$ in $\mathbb{D}$.  Similarly, $\bar V^r\Rightarrow \mu e$ in $\mathbb D$.  Now compute
\begin{equation*}
\begin{split}
\bar R^r(t)&=\frac{1}{r}\sup\left\{m\geq0:V^r(m)+H(U^r,V^r,1)(m)\leq rt\right\}\\
&=\sup\left\{x/r\geq 0:V^r(x)+H(U^r,V^r,1)(x)\leq rt\right\}\\
&=\sup\left\{x/r\geq 0:\frac{V^r(x)}{r}+\frac{1}{r}H(U^r,V^r,1)(x)\leq t\right\}\\
&=\sup\left\{y\geq 0:\frac{V^r(ry)}{r}+\frac{1}{r}H(U^r,V^r,1)(ry)\leq t\right\}\\
&=\sup\left\{y\geq 0:\bar V^r(y)+H(\bar U^r,\bar V^r,1/r)(y)\leq t\right\},
\end{split}
\end{equation*}
by lemma $\ref{HScale}$.  We have $(\bar U^r,\bar V^r,1/r)\Rightarrow (\mu
e,\mu e, 0)$ in $\mathbb{D}$ since the processes are independent.  The
function $H$ is continuous at $(\mu_u e,\mu_v e, 0)$, and addition is
continuous at continuous elements of $\mathbb D$, so 
\begin{equation*}
\bar V^r+H(\bar U^r,\bar V^r,1/r)\Rightarrow \mu e
\end{equation*} 
in $\mathbb{D}$.  The result follows because $\mu e$ is in the set of
continuity for the function $x\mapsto \sup\{y\geq 0:x(y)\leq t\}$ by Corollary
13.6.4 in \cite{whitt2002stochastic}.
\end{proof}

We now prove the main result.

\noindent{\it Proof of Theorem \ref{Main Theorem}.} By Lemma \ref{lem:M=F} 
\begin{equation*}
M(t)=F(U^r,V^r,1)(R(t)).
\end{equation*}
Under fluid scaling $\bar R^r\Rightarrow e/\mu$ by \ref{lem:RnScale}.  We
first consider the scaling limit for $F$, before composing with $R$.
\begin{equation*}
\begin{split}
a_r^{-1}F(U^r,&V^r,1)(rt)=a_r^{-1}\sup_{0\leq s\leq rt}\left(V^r(s)-V^r(s-)+H(U^r,V^r,1)(s)\right)\\
&\qquad\qquad-a_r^{-1}H(U^r,V^r,1)(rt)\\
&=\sup_{0\leq s\leq rt}\left(a_r^{-1}V^r(s)-a_r^{-1}V^r(s-)+a_r^{-1}H(U^r,V^r,1)(s)\right)\\
&\qquad\qquad-a_r^{-1}H(U^r,V^r,1)(rt)\\
&=\sup_{0\leq s\leq t}\left(a_r^{-1}V^r(rs)-a_r^{-1}V^r(rs-)+a_r^{-1}H(U^r,V^r,1)(rs)\right)\\
&\qquad\qquad-a_r^{-1}H(U^r,V^r,1)(rt).
\end{split}
\end{equation*}
$t\mapsto r\nu^rt$ is continuous so $r\nu^r(rs)-r\nu^r(rs-)=0$ and
\begin{equation}\label{express F}
\begin{split}
a_r^{-1}F(U^r,V^r,1)(rt)&=\sup_{0\leq s\leq t}\left(\check V^r(s)-\check V^r(s-)+a_r^{-1}H(U^r,V^r,1)(rs)\right)\\
&\qquad\qquad-a_r^{-1}H(U^r,V^r,1)(rt).
\end{split}
\end{equation}
Now, we address the idleness part of \eqref{express F} that occurs twice.

\begin{equation*}
\begin{split}
a_r^{-1}&H(U^r,V^r,1)(rt)\\
&=a_r^{-1}\left.\left.\sup_{0\leq s\leq rt}\right(U^r(s)-V^r([s-1]^+)\right)\\
&=\left.\left.\sup_{0\leq s\leq t}\right(a_r^{-1}U^r(rs)-a_r^{-1}V^r(r[s-1/r]^+)\right)\\
&=\sup_{0\leq s\leq t}\bigg(a_r^{-1}\left(U^r(rs)-r\mu^rs\right)+a_r^{-1}r\mu^rs\\
&\qquad-a_r^{-1}\left(V^r(r[s-1/r]^+)-r\nu^r[s-1/r]^+\right)-a_r^{-1}r\nu^r[s-1/r]^+\bigg)\\
&=\sup_{0\leq s\leq t}\bigg(\check U^r(s)+a_r^{-1}r\mu^r s -\check V^r([s-1/r]^+)-a_r^{-1}r\nu^r[s-1/r]^+\bigg)\\
&=\sup_{0\leq s\leq t}\bigg(\check U^r(s)+a_r^{-1}r(\mu^r-\nu^r)s+a_r^{-1}r\nu^r(s-[s-1/r]^+)\\
&\qquad -\check V^{r}([s-1/r]^+)\bigg).
\end{split}
\end{equation*}
Since 
\begin{equation*}
a_r^{-1}r\nu^r(s-[s-1/r]^+)=a_r^{-1}r\nu^r(1/r\wedge s)=a_r^{-1}\nu^r(1\wedge rs),
\end{equation*}
 we have
\begin{equation*}
\begin{split}
a_r^{-1}&H(U^r,V^r,1)(rt)\\
&=H(\check U^r+a_r^{-1}r(\mu^r-\nu^r)e+a_r^{-1}\nu^r(1\wedge re),\check V^{r},1/r)(t).
\end{split}
\end{equation*}

Putting this expression back into \eqref{express F},
\begin{equation*}
\begin{split}
a_r^{-1}&F(U^r,V^r,1)(rt)=\sup_{0\leq s\leq t}\left[\check V^r(s)-\check V^r(s-)\right.\\
&\left.+H(\check U^r+a_r^{-1}r(\mu^r-\nu^r)e+a_r^{-1}\nu^r(1\wedge re),\check V^{r},1/r)(s)\right]\\
&-H(\check U^r+a_r^{-1}r(\mu^r-\nu^r)e+a_r^{-1}\nu^r(1\wedge re),\check V^{r},1/r)(t)\\
&=F(\check U^r+a_r^{-1}r(\mu^r-\nu^r)e+a_r^{-1}\nu^r(1\wedge re),\check V^r,1/r)(t).
\end{split}
\end{equation*}
By Lemma \ref{JointConvergence} we have $(U^*+\gamma\mu e,V^*,0)$ satisfies
the continuity criterion of Lemma \ref{FisContinuous}.  By the continuous
mapping theorem
\begin{equation*}
F(\check U^r+a_r^{-1}r(\mu^r-\nu^r)e+a_r^{-1}\nu^r(1\wedge re),\check V^r,1/r)\Rightarrow F(U^*+\gamma\mu e,V^*,0).
\end{equation*}
Finally, the scaled plateau process is a composition of $F$ with $R$,
\begin{equation*}
a_r^{-1}F(U^r,V^r,1)(R(rt))=a_r^{-1}F(U^r,V^r,1)(r\bar R^r(t)).
\end{equation*}
Composition is continuous on $(\mathbb{D}\times C_0)$ by \cite{whitt1980some}
Theorem 3.1, where $C_0\subset \mathbb D$ denotes the strictly increasing,
continuous functions.  So the continuous mapping theorem yields
\begin{equation*}
a_r^{-1}M^r(r\cdot)=\check M^r\Rightarrow M^*= F(U^*+\gamma \mu e,V^*,0)(\cdot/\mu).
\end{equation*}
{\hfill$\blacksquare$}

\section{Analysis of the limit process}\label{s.LimitAnalysis}

In this section we derive, for certain cases, some properties of the
stochastic process $M^*$ that appears as the scaling limit of the plateau
process. We focus on cases where the interarrival time distribution has finite
variance (but the service time distribution still has infinite variance),
leading to a trivial limit for the arrival process $U^*(t)\equiv 0$ and a
non-trivial $\alpha$-stable process $V^* $ for the limit of the service
process. By Theorem \ref{Main Theorem}, the limit of the plateau process is
then 
\[ M^*(t)= F(\gamma\mu e,V^*,0)(t/\mu), \qquad t\ge0. \]

Although this process is not Markov, a suitable time change of it is and has
one-dimensional distributions that can be derived explicitly. The time change
is simply an inverse local time of the reflected (at zero) version of the
process $V^*(t)-\gamma\mu t$, $t\ge0$. More explicitly, letting $X(t)
= V^*(t)-\gamma \mu t$ and  $\underline{X}(t)=\inf_{0\le s\le t}X(s)$, the
process $L(t)=-\underline{X}(t)$ is the local time at zero for the reflected
process $Y(t)=X(t)-\underline{X}(t)$ associated with $X$. We use its
right-continuous inverse $L^{-1}$ to define the time-changed version
\[  Z(v)=M^*(\mu L^{-1}(v)), \qquad v\ge 0, \]
of our limit process. The one-dimensional distributions of $Z$ are given by
the following. 

\begin{thm}\label{thrm2main}
If the limiting arrival process is identically zero, then for each $v\ge0$,
the distribution function $F_v$ of $Z(v)$ is given by
\begin{equation} \label{e.thrm2main}
  F_v(y)=\exp\left(-\int_y^{y+v}\frac{\kappa(q)}{q}dq\right), \qquad y\ge0,
\end{equation}
where $\kappa(q)/q=\phi_q(c_\alpha q^{-\alpha}/\alpha)$, $\phi_q$ is the
right-inverse of 
\[ s\mapsto s+s^\alpha + c_\alpha
\int_q^\infty (1-e^{-sx}) x^{-\alpha-1 } dx, \]
and $c_\alpha$ is an explicit constant (see below). 
\end{thm}

Not only does this result provide some means to perform calculations on the
process $Z$ (and thus on the process $M^*)$, but it also allows us to relate
Theorem \ref{Main Theorem} to the results
obtained in \cite{GTZ1}. In particular, comparing with Theorem 2.2 in
\cite{GTZ1}, we see that the above one-dimensional distributions of our time
changed limit process $Z(v)=M^*(\mu L^{-1}(v))$ are precisely the limiting
laws of the one-dimensional distributions of the process studied in
\cite{GTZ1} (a discrete-time Markov chain embedded in the plateau process), in
which the analogous time change was performed on the original (prelimit)
process before scaling and taking the limit. 

The remainder of this section provides the proof. 

\subsection{Proof of Theorem \ref{thrm2main}}\label{ss.excursionTheory}

Since $U^*\equiv 0$, we are using the function
\begin{equation*}
F(\gamma\mu e,y,0)(t/\mu) = \sup_{0\leq s\leq t/\mu}[y(s)-y(s-)+\sup_{0\leq r\leq
s}[\gamma\mu r-y(r)]]-\sup_{0\leq s\leq t/\mu} [\gamma\mu s-y(s)],
\end{equation*}
where $y$ is replaced by the $\alpha$-stable process $V^*$. Using the
definition of $X(t)$ and $\underline{X}(t)$, this expression reduces to
\begin{equation}\label{e.Membedded}
  \begin{aligned}
  F(\gamma\mu e,V^*,0)(t/\mu) &=\sup_{0\leq s\leq t/\mu}[X(s)-X(s-)-\inf_{0\leq
  r\leq s} X(r)]+\inf_{0\leq s\leq t/\mu} X(s) \\
 &=  \sup_{0\leq s\leq t/\mu}[X(s)-X(s-)+ \underline{X}(t/\mu) - \underline{X}(s)].
\end{aligned}
\end{equation}
Recall that $L(t)=-\underline{X}(t)$ is the local time at zero for the reflected
process $Y(t)=X(t)-\underline{X}(t)$ associated with $X$. For $v\ge0$ define
\begin{equation*}
Z(v) = \sup_{0\leq s\leq L^{-1}(v)}[X(s)-X(s-) -  (v - L(s))],
\end{equation*}
where $L^{-1}$ denotes the right-continuous inverse.
Then from \eqref{e.Membedded} we see that for times $t$ such that
$L(t/\mu)=v$, $M^*(t)=Z(v)$. Put another way, we have for all $t\ge0$ that
$M^*(\mu L^{-1}(L(t/\mu)))=Z(L(t/\mu))$. That is, the process $Z(v)=M^*(\mu
L^{-1}(v))$ is a certain time-changed (and embedded) version of the process
$M^*$, evaluated at times (scaled by $\mu$) when the local time of $Y$
has attained the level $v$. We now examine the one-dimensional distributions
of the process $Z$. 

For each $v\ge0$ we will derive the distribution function $F_v(y)$, $y\ge0$,
of $Z(v)$ using some calculations from excursion theory (note that $Z(v)$ is a
nonnegative random variable).  As in Chapter 4 of Bertoin \cite{Bertoin1996},
define  $N = ((v,\varepsilon(v)), v\geq 0)$ as the Poisson point process of
excursions away from $0$ for the reflected process $Y$. That is,
$(v,\varepsilon(v))$ takes values in $[0,\infty)\times\mathscr{E}$, where
  $\mathscr{E}$ is the space of excursions from zero, and $\varepsilon(v)$
  corresponds to the excursion of $Y$ beginning when its local time has
  attained level $v$.  Let $\ell$ denote Lebesgue measure and denote by $n$
  the excursion measure of $Y$, which is the sigma-finite measure on
  $\mathscr{E}$ such that $\ell\times n$ is the intensity on
  $[0,\infty)\times\mathscr{E}$ of the Poisson random measure $N$.  

Defining $\Delta(v)= \Delta(\varepsilon(v))$ to be the largest jump made
during the excursion $\varepsilon(v)$ (which we set to be 0 if there is no
excursion at $v$), we see that \begin{equation} Z(v) = \sup_{0\leq u\leq v}
  [\Delta (u) - (v-u)].  \end{equation}
Since $N' = \left(\sum_{v} \delta_{(v,\Delta(v))}, v\geq 0\right)$ is a
Poisson point process on $[0,\infty)\times[0,\infty)$, the process $Z$ is
  Markov.  Note that for any $w\in[0,v]$,
\begin{eqnarray*}
Z(v) &=& \max\{ \sup_{0\leq u\leq w} [\Delta (u) - (w-u)] - (v-w), \sup_{w\leq u\leq v}[\Delta (u) - (v-u)]\}\\
 &=& \max\{ Z(w) - (v-w), \sup_{w\leq u\leq v}[\Delta (u) - (v-u)]\}\\
 &\sim&  \max\{ Z(w) - (v-w), \sup_{u \in [0, v-w]} [\Delta (u) -
 u]\}.
\end{eqnarray*}
In particular, taking
$w=0$, we obtain
\begin{equation}
Z(v)  \sim  \sup_{0\leq u\leq v} [\Delta (u) - u].
\end{equation}

Define $A=A_{v,y} = \{(u,\varepsilon)\in[0,\infty)\times\mathscr{E}: u\in
[0,v],\Delta(\varepsilon) > y+u\}$. Then using standard results (e.g.
Section 0.5 of Bertoin \cite{Bertoin1996}), we see that for $y>0$,
\begin{equation}
P(Z(v)>y) = P(N(A) \geq 1).
\end{equation} 
The random variable $N(A)$ is Poisson with mean 
\begin{equation}\label{e.LambdaFormula}
  \lambda(v,y)=(\ell\times n)(A)=\int_0^v n(\Delta (\varepsilon)> y + u)du
  =\int_y^{y+v} n(\Delta (\varepsilon)> q)dq. 
\end{equation}
So the distribution function of $Z(v)$
is $F_v(y)=\exp(-\lambda(v,y))$, $y>0$, which is explicit as long as we can derive an
expression for $n(\Delta (\varepsilon)>q)$ for each $q>0$. 

To this end, fix $q>0$. The idea is to compare the set of excursions with a
jump bigger than $q$ to the set of excursions of a modified process, whose
lifetimes are longer than the exponential waiting time until the first
$q$-jump of the original process. The modified process $\tilde{Y}$ is obtained
from $Y$ by thinning all jumps of size greater than $q$, yielding a L\'evy
process for which the L\'evy measure is now restricted to $[0,q]$, so that we
may apply a formula of Baurdoux \cite{Baurdoux:2009} for excursion lifetimes.  

In more detail, write $X=\tilde{X}+J_q$, where $J_q$ is a pure jump process
independent of $\tilde{X}$ with all jumps greater than $q$, and $\tilde{X}$
almost surely has all jumps bounded by $q$. Define the modified process
$\tilde{Y}(t)=\tilde{X}(t)-\underline{\tilde{X}}(t)$ and let $\tilde{n}$
denote the excursion measure on $\mathscr{E}$ of the process $\tilde{Y}$. 

The Laplace exponent of the
L\'evy process $X$ is $\Psi(s) = s+ s^\alpha$, and the corresponding L\'evy
measure $\nu(dx) = c_\alpha x^{-\alpha-1} dx$, for a strictly positive
constant $c_\alpha$ (an expression is given in Exercise 1.4 of
\cite{Kyprianoubook}). So the L\'evy measures of $\tilde{X}$ and $J_q$ are
$\nu$ restricted to $[0,q]$ and $(q,\infty)$ respectively. The L\'evy exponent
of $\tilde{X}(t)$ can be written as 
\begin{equation}\label{e.ModLevyExp}
  \tilde{\Psi}_q(s) = s+s^\alpha + c_\alpha
\int_q^\infty (1-e^{-sx}) x^{-\alpha-1 } dx.
\end{equation}

Define
\[ \mathrm{e}_q=\inf\{t\ge 0:\ Y(t)-Y(t-)>q\}
  \]
as the waiting time until the first jump of $Y$ of size greater than
$q$. Then $\mathrm{e}_q=\inf\{t\ge0:J_q(t)>J_q(t-)\}$, and since $J_q$ is independent of $\tilde{X}$, the random variable
$\mathrm{e}_q$ is exponential with rate $\beta_q=\nu(q,\infty)=c_\alpha
q^{-\alpha}/\alpha$ and is independent of $\tilde{X}$. 

\begin{lem}\label{l.comparison} For each $q>0$,
  \begin{equation}\label{e.comparison}
  n(\Delta (\varepsilon)>q) =
  \int_{\mathscr{E}}(1-e^{-\beta_q|\varepsilon|})d\tilde{n}(\varepsilon),
\end{equation}
where $|\varepsilon|$ denotes the lifetime of an excursion
$\varepsilon\in\mathscr{E}$. 
\end{lem}

\begin{proof}
  We show that both expressions are equal to $1/E[L(\mathrm{e}_q)]$. Beginning
  with the left side, multiply and divide by $E[L(\mathrm{e}_q)]$ to obtain
 \[ \begin{aligned}
    n(\Delta (\varepsilon)>q) & =
    \int_{\mathscr{E}}
    1_{\{\Delta(\varepsilon)>q\}}dn(\varepsilon)\\
    & = \frac{1}{E[L(\mathrm{e}_q)]}E\left[ \int_0^\infty\int_{\mathscr{E}}
    1_{\{\Delta(\varepsilon)>q\}} 1_{\{s\le \mathrm{e}_q\}}
    dn(\varepsilon)dL(s)\right].
  \end{aligned}\]
We show the second expectation on the right equals one.  Since the function
$G(s,\omega,\varepsilon)= 1_{\{\Delta(\varepsilon)>q\}} 1_{\{s\le
  \mathrm{e}_q(\omega)\}}$ on $[0,\infty)\times\Omega\times\mathscr{E}$ is
    measurable and almost surely left-continuous in $s$, the compensation
    formula in excursion theory (see Corollary 11 in Section IV.4 of
    \cite{Bertoin1996}) yields
\[
E\left[ \int_0^\infty\int_{\mathscr{E}}
    1_{\{\Delta(\varepsilon)>q\}} 1_{\{s\le \mathrm{e}_q\}}
    dn(\varepsilon)dL(s)\right]
  =E\left[\sum_g  1_{\{\Delta(\varepsilon_g)>q\}} 1_{\{g\le \mathrm{e}_q\}}
    \right],
    \]
where for each sample path, the sum is over the left endpoints of all excursion intervals
$(g,d)$ and $\varepsilon_g$ is the excursion of $Y$ beginning at time $g$. But
since $\mathrm{e}_q$ falls during the first excursion with a jump greater than
$q$, the sum equals one almost surely. 

Turning to the right side of \eqref{e.comparison}, we again multiply and
divide, noting that $E[L(\mathrm{e}_q)]=E[\tilde{L}(\mathrm{e}_q)]$, where
$\tilde{L}(t)=-\underline{\tilde{X}}(t)$ is the local time for $\tilde{Y}$,
because the sample paths of $Y$ and $\tilde{Y}$ are identical up to time
$\mathrm{e}_q$. This gives
\[
   \int_{\mathscr{E}}(1-e^{-\beta_q|\varepsilon|})d\tilde{n}(\varepsilon)
   =\frac{1}{E[L(\mathrm{e}_q)]}
   E\left[ \int_0^\infty\int_{\mathscr{E}}(1-e^{-\beta_q|\varepsilon|})
     1_{\{s\le\mathrm{e}_q\}}d\tilde{n}(\varepsilon)d\tilde{L}(s)\right],
  \]
and we must show the second expectation on the right equals one. Using the
compensation formula,
\[
   E\left[ \int_0^\infty\int_{\mathscr{E}}(1-e^{-\beta_q|\varepsilon|})
     1_{\{s\le\mathrm{e}_q\}}d\tilde{n}(\varepsilon)d\tilde{L}(s)\right]
     = E\left[ \sum_g(1-e^{-\beta_q|\varepsilon_g|})1_{\{g\le\mathrm{e}_q\}}
       \right],
     \]
where this time the sum is over all excursion intervals $(g,d)$ of
$\tilde{Y}$ and $\varepsilon_g$ are the corresponding excursions. Since
$\mathrm{e}_q$ is independent of $\tilde{Y}$, the expectation on the right can
be computed as an iterated integral over $\mathbb{D}\times[0,\infty)$ with respect
  to the product law $P_{\tilde{Y}}\times P_e$ of the random pair
  $(\tilde{Y},\mathrm{e}_q)$. This yields 
  \begin{equation}
    \begin{aligned}
    \nonumber
    E_{\tilde{Y}} E_e\left[ \sum_g(1-e^{-\beta_q|\varepsilon_g|})1_{\{g\le\mathrm{e}_q\}}
       \right]
       & =  E_{\tilde{Y}} \left[ \sum_g(1-e^{-\beta_q|\varepsilon_g|})
         P_e(g\le\mathrm{e}_q)\right] \\
       & = E_{\tilde{Y}} \left[ \sum_g
         P_e(|\varepsilon_g|>\mathrm{e}_q)
         P_e(\mathrm{e}_q>g) \right].
    \end{aligned}
  \end{equation}
Note that for each excursion interval $(g,d)$, the lifetime
$|\varepsilon_g|=d-g$. So by the memoryless property of the exponential and
since the excursion intervals are disjoint, the right side above is equal to
\begin{equation}
\begin{aligned}
  \nonumber
   E_{\tilde{Y}} \left[ \sum_g
     P_e(\mathrm{e}_q<d\,\big\vert\,\mathrm{e}_q>g)
         P_e(\mathrm{e}_q>g) \right]
         & =  E_{\tilde{Y}} \left[ \sum_g
           P_e\left(\mathrm{e}_q\in (g,d)\right) \right]\\
          & = E_{\tilde{Y}} \left[ P_e\left(\mathrm{e}_q\in
            [0,\infty)\setminus \overline{\mathscr{Z}}\right) \right],
\end{aligned}
\end{equation}
where $\overline{\mathscr{Z}}$ denotes the closure of the zero set of
$\tilde{Y}$. Since $X=V^*-\gamma\mu t$ is not a monotone or pure jump process,
this set has Lebesgue measure zero and the right side above equals one.
\end{proof}

Since the L\'evy measure of $\tilde{X}$ has bounded support, we can apply
Equation (3.3) of \cite{Baurdoux:2009} to the right side of
\eqref{e.comparison}, which in the notation of \cite{Baurdoux:2009} would be
written ``$\tilde{n}(|\varepsilon|>\mathrm{e}_q)$.'' Let $P_x$ denote the law
of $\tilde{X}+x$ and $\tau^x_q=\inf\{t\ge0:\tilde{X}(t)+x=0\}$ be the hitting
time of zero. Then \eqref{e.comparison} combined with \cite{Baurdoux:2009}
Equation (3.3) in our setting (in particular $h(x)$ there is simply $x$ here)
yields
\begin{equation}\nonumber
  n(\Delta (\varepsilon)>q) = \lim_{x\downarrow 0} \frac{P_x(\tau^x_q >
    \mathrm{e}_q)}{x} = \lim_{x\downarrow 0} \frac{1- E_x[e^{-\beta_q \tau^x_q}]}{x}.
\end{equation}
Observe that
\begin{equation}\nonumber
E_x[e^{-\beta_q \tau^x_q}] = e^{-x \phi_q(\beta_q)},
\end{equation}
with $\phi_q$ the right inverse of $\tilde{\Psi}_q$.

Thus, we obtain
\begin{equation}\label{e.compB}
n(\Delta (\varepsilon)>q) = \phi_q(\beta_q) = \phi_q(c_\alpha q^{-\alpha}/\alpha) =: h(q).
\end{equation}
Rewrite the last expression using \eqref{e.ModLevyExp} to get
\begin{equation}\nonumber
  c_\alpha q^{-\alpha}/\alpha = \tilde{\Psi}_q(h(q)) = h(q)+h(q)^\alpha + c_\alpha \int_q^\infty (1-e^{-h(q)x}) x^{-\alpha-1 } dx.
\end{equation}
This can be simplified to
\begin{equation}\nonumber
  h(q)+ h(q)^\alpha = c_\alpha \int_q^\infty e^{-h(q)x} x^{-\alpha-1 } dx.
\end{equation}
Defining $\kappa(q) = h(q)q$, performing a change of variables $t=x/q$, and letting $T_\alpha$ be a Pareto distributed random variable with index $\alpha$ we obtain
\begin{equation}\label{e.Last}
  q^{\alpha-1} \kappa(q) + \kappa(q)^\alpha = \frac{c_\alpha}{\alpha} E[e^{-\kappa(q) T_\alpha}].
\end{equation}
This equation can be transformed into Equation (7) in \cite{GTZ1} for $\kappa(y)$ 
(using $\lambda=1$ and $\gamma=-\Gamma(1-\nu)$
there), and so we see by Lemma 3.9 of \cite{GTZ1} that \eqref{e.Last} has a unique
solution $\kappa(q)$, which by Lemma 3.11 in \cite{GTZ1} is a continuous,
bounded, regularly varying function of $q$ with index $1-\alpha$.
Combining \eqref{e.LambdaFormula} with \eqref{e.compB} establishes
\eqref{e.thrm2main} and proves Theorem \ref{thrm2main}.

\bibliographystyle{acm}
\bibliography{GTZ}

\vspace{2ex}
\begin{minipage}[t]{2in}
\footnotesize {\sc Department of Mathematics\\
University of Virginia\\
Charlottesville, VA 22904\\
E-mail: gromoll@virginia.edu\\
E-mail: bat5ct@virginia.edu}
\end{minipage}
\hspace{5ex}
\begin{minipage}[t]{4in}
\footnotesize {\sc Centrum Wiskunde \& Informatica\\
P.O. Box 94079\\
1090 GB Amsterdam, Netherlands\\
E-mail: Bert.Zwart@cwi.nl} 
\end{minipage}

\end{document}